\documentclass[amsthm]{amsart}
\usepackage{amsmath,amsfonts,amssymb}
\usepackage[usenames]{color}
\usepackage{srcltx}
\usepackage{hyperref}
\usepackage{algorithm}
\usepackage{algorithmic}
\usepackage{graphicx}
\usepackage{enumerate}

\def\NP{{\mathbf{NP}}}

\def\MN{{\mathbb{N}}}
\def\MZ{{\mathbb{Z}}}

\newtheorem{theorem}{Theorem}
\newtheorem{lemma}[theorem]{Lemma}
\newtheorem{proposition}[theorem]{Proposition}
\newtheorem{corollary}[theorem]{Corollary}
\newtheorem{remark}[theorem]{Remark}

\theoremstyle{definition}

\newcommand{\gpr}[2]{{\left\langle #1 \mid #2 \right\rangle}}

\let\oldmarginpar\marginpar
\renewcommand\marginpar[1]{\-\oldmarginpar[\raggedleft\footnotesize #1]%
{\raggedright\footnotesize #1}}

\DeclareMathOperator{\SL}{SL}

\def\ovx{{\overline{x}}}

\DeclareMathOperator{\Mat}{{Mat}}

\def\P{{\mathbf{P}}}

\def\NP{{\mathbf{NP}}}
\def\SSP{{\mathbf{SSP}}}
\def\ZOE{{\mathbf{ZOE}}}

{}
{}

\title{Subset sum problem in polycyclic groups}

\author[]{Andrey Nikolaev and Alexander Ushakov}

\address{Andrey Nikolaev, Stevens Institute of Technology, Hoboken, NJ, 07030 USA}
\email{anikolae@stevens.edu}

\address{Alexander Ushakov, Stevens Institute of Technology, Hoboken, NJ, 07030 USA}
\email{aushakov@stevens.edu}

\begin{document}

	
%

\maketitle

\begin{abstract} We consider a group-theoretic analogue of the classic subset sum problem. 
It is known that every virtually nilpotent group has polynomial time decidable subset sum problem.
In this paper we use subgroup distortion to show that every polycyclic non-virtually-nilpotent 
group has $\NP$-complete subset sum problem.

\noindent
{\bf Keywords.}
Polycyclic groups, nilpotent groups, subset sum problem, $\NP$-completeness.

\noindent
{\bf 2010 Mathematics Subject Classification.} 03D15, 20F65, 20F10, 20F16.

\end{abstract}
\maketitle

\section{Introduction}\label{sec:intro}

The study of discrete optimization problems in groups was initiated in~\cite{Miasnikov-Nikolaev-Ushakov:2014a}, 
where the authors introduced group-theoretic generalizations of the classic knapsack problem 
and its variations, e.g., subset sum problem and bounded submonoid membership problem. 
In the subsequent papers~\cite{Miasnikov-Nikolaev-Ushakov:2014b} and~\cite{Myasnikov-Nikolaev-Ushakov:2016}, the authors studied generalizations of the Post corresponce problem and classic lattice problems in groups.
The investigation of knapsack-type problems in groups continued in 
papers~\cite{Frenkel-Nikolaev-Ushakov:2014}, \cite{Konig-Lohrey-Zetzsche:2015}, \cite{Lohrey-Zetzsche:2015}, \cite{Misch-Tr:2016a}, \cite{Misch-Tr:2016b}.
The computational properties of these problems, 
aside from being interesting in their own right, 
were shown to be closely related to a wide range of well-known geometric and algorithmic properties of groups. 
For instance, the complexity of knapsack-type problems 
in certain groups depends on geometric features 
of a group such as growth, subgroup distortion, and negative curvature. 
The Post correspondence problem in $G$ is closely related to twisted conjugacy problem in $G$, equalizer problem in $G$, and a strong version of the word problem. Furthermore, lattice problems are related to the classic subgroup membership problem
and finite state automata. We refer the reader to the aforementioned papers for details.

In this paper, we prove $\NP$-completeness of the subset sum problem in any virtually polycyclic non-virtually-nilpotent group by exploiting properties of exponentially distorted subgroups. This highlights a notable connection between geometric and combinatorial properties of polycyclic groups.

\subsection{Subset sum problem}

Let $G$ be a group generated by a finite set 
$X=\{x_1,\ldots,x_n\}\subseteq G$. Elements in $G$ can be expressed
as products of the generators in $X$ and their inverses.
Hence, we can state the following combinatorial problem.

\medskip
\noindent{\bf The subset sum  problem $\SSP(G,X)$\index{$\SSP(G,X)$}:}
Given $g_1,\ldots,g_k,g\in G$ decide if
  \begin{equation} \label{eq:SSP-def}
  g = g_1^{\varepsilon_1} \ldots g_k^{\varepsilon_k}
  \end{equation}
for some $\varepsilon_1,\ldots,\varepsilon_k \in \{0,1\}$.

\medskip
By \cite[Proposition 2.5]{Miasnikov-Nikolaev-Ushakov:2014a}
computational properties of $\SSP$ do not depend on the choice of a finite generating set $X$
and, hence, the problem can be abbreviated as $\SSP(G)$.
Also, the same paper provides a variety of examples
of groups with  $\NP$-complete (or polynomial time) subset sum problems.
For instance, $\SSP$ is $\NP$-complete for the following groups:
\begin{enumerate}[(a)]
\item 
abelian group $\MZ^\omega$;
\item 
free metabelian non-abelian groups;
\item 
wreath products of finitely generated infinite abelian groups;
\item 
metabelian Baumslag--Solitar groups $BS(m, n)$ with $0\ne m\ne n\ne 0$;
\item 
metabelian group
$GB =\gpr{a,s,t}{[a, a^t ] = 1, [s, t] = 1, as = aa^t}$;
\item 
Thompson's group $F$.
\end{enumerate}
One can observe that in a number of the above examples, $\NP$-completeness of $\SSP$ is 
a consequence of exponential subgroup distortion. Recently, K\"onig, Lohrey, and Zetzsche 
in~\cite{Konig-Lohrey-Zetzsche:2015} 
found a polycyclic group $G$ with $\NP$-complete $\SSP$. Upon inspection of the proof, one can notice 
that exponential subgroup distortion plays a key role in their argument, as well. 
In the present work, we show that this is not coincidental.
Specifically, in the case of polycyclic groups, we investigate subgroup distortion 
to give a complete description of $\SSP$. We determine out that $\SSP$ in a polycyclic group 
is $\NP$-complete if the group is not virtually nilpotent, and polynomial time decidable if it is. The latter is known by~\cite{Miasnikov-Nikolaev-Ushakov:2014a}; it is also shown by~\cite{Konig-Lohrey-Zetzsche:2015} that for virtually nilpotent groups, $\SSP$ is decidable in nondeterministic logspace. Our work is heavily inspired by results of~\cite{Osin:2002}, even though we do not, strictly speaking, rely on them.


The authors express their gratitude to A.~Miasnikov and M.~Sohrabi 
for insightful discussions and advice.

\subsection{Zero-one equation problem}\label{sub:zoe}

Recall that a vector $v \in \MZ^n$ is called a {\em zero-one} vector
if each entry in $v$ is either $0$ or $1$.
Similarly, a square matrix  $A\in \Mat(n,\MZ)$ is called a {\em zero-one}
matrix if each entry in $A$ is either $0$ or $1$.
Let $1^n$ denote the vector $(1,\ldots,1) \in \MZ^n$.
The following problem is $\NP$-complete  
(see \cite[Section 8.3]{Dasgupta-Papadimitriou-Vazirani:2006}).

\medskip\noindent
{\bf Zero-one equation problem (ZOE):}
Given a zero-one matrix $A \in \Mat(n,\MZ)$ decide
if there exists a zero-one vector $\ovx \in \MZ^n$ satisfying
$A\cdot \ovx = 1^n$, or not.

\section{Preliminary results in nilpotent groups}
\label{th:norm_forms}

In this section, we prove a technical result on nilpotent groups (Proposition~\ref{pr:permutation}) 
that we use later in the proof of $\NP$-hardness in Section~\ref{sec:general_case}.


Let $H$ be a finitely generated nilpotent group. A choice of a Malcev basis provides normal forms for elements in $H$. Furthermore, it is well known that the length
of the normal form for the element $w$ is bounded by $p(|w|)$
where $p$ is a polynomial (see, for example,~\cite{Hal69}). 
The polynomial~$p$
depends on~$H$ and the choice of a Malcev basis.
Below we establish a similar uniform 
(with respect to the number of generators) 
bound for all finitely generated free nilpotent groups of class $c$. 

Let $H=N_{r,c}$ be the free nilpotent group of rank $r$ and class $c$.
Put $y_{11}=x_1,\ldots, y_{1r}=x_r$. Further, let $y_{i1},\ldots,y_{ij_i}$ 
be the iterated commutators of weight $i$ that constitute a basis for the $i$th quotient 
of the derived series of $H$, $\gamma_{i}(H)/\gamma_{i+1}(H)$.  
For convenience rename them $y_1,\ldots,y_m$ so that this list is ordered by 
weight and $y_1=x_1,\ldots, y_r=x_r$. This tuple is called a 
\emph{Malcev basis of iterated commutators} for $H$. 
Below we use the enumerations $y_{ij}$ and $y_i$ interchangeably, 
as convenient in each particular case.

The following theorem is well known in the case of a fixed rank. 
The proof below essentially repeats the one in~\cite{MMNV} 
(which, in turn, is adapted from unpublished lecture notes by 
C.~Dru\c{t}u and M.~Kapovich), accounting for a variable number of generators. The key observation in the proof is that commuting iterated commutators of weights $a$ and $b$ produces iterated commutators of weights at least $a+b$. Therefore, each particular commutator can be iterated during the collection process not more than $c$ times (less depending on the involved weights), which puts a bound on the number of iterated commutators of each particular weight in the final expression.

\begin{theorem}[Uniform length bound for free nilpotent groups]
\label{th:nilp_growth} 
For each integer $c\ge 1$ there exists a polynomial $P$ for which the following takes place. Let $N_{r,c}$ be the free nilpotent group of rank $r$ and class $c$, and let $w$ be a group word in generators $x_1,\ldots,x_r$ of $N_{r,c}$. Let $y_1,\ldots,y_m$ be a Malcev basis of iterated commutators for $N_{r,c}$. Then there are integers $\alpha_1,\ldots,\alpha_m$ such that
\begin{equation}\label{eq:malcev_coords}
w=y_1^{\alpha_1}\cdots y_m^{\alpha_m}
\end{equation}
in $N_{r,c}$ and $|\alpha_i|\le P(|w|)$ for $i=1,\ldots,m$.
\end{theorem}
\begin{proof}
Consider the free nilpotent group $N_{2c,c}$ generated by $z_1,\ldots,z_c$ and $w_1,\ldots,w_c$ and all possible expressions
\begin{equation}\label{eq:commutator_z}
t's'=s't'\cdot u',
\end{equation}
where
\[
s'=[\ldots[z_{1}^{\pm1},z_{2}^{\pm1}],\ldots,z_{k}^{\pm1}]^{\pm1}\mbox { and } t'=[\ldots[w_1^{\pm1},w_{2}^{\pm1}],\ldots,w_{l}^{\pm1}]^{\pm1},
\]
with $u'$ expressed as a product of commutators of weight at least $k+l$. Let $C_0$ be the maximum of the finitely many possible lengths of $u'$ for all $1\le k<l\le c$.

For such constant $C_0$ (note that it depends only on the nilpotency class $c$) the following takes place. Let $s=y_i^{\pm 1}\in N_{r,c}$ and $t=y_j^{\pm 1}\in N_{r,c}$ be of weights $k<l$, respectively,
\[
s=[\ldots[x_{i_1}^{\pm1},x_{i_2}^{\pm1}],\ldots,x_{i_k}^{\pm1}]^{\pm1}\mbox { and } t=[\ldots[x_{j_1}^{\pm1},x_{j_2}^{\pm1}],\ldots,x_{j_l}^{\pm1}]^{\pm1}.
\]
Specializing equality~\eqref{eq:commutator_z} by $z_{\nu}\to x_{i_\nu}$ and $w_{\nu}\to x_{j_\nu}$ we obtain that
\begin{equation}\label{eq:commutator}
ts=st\cdot u,
\end{equation}
where $u$ is word in iterated commutators $y_i$ in $N_{r,c}$ of weight at least $k+l$ and length at most~$C_0$.

Now, we turn to processing the word $w$. For a group word $v$ in the alphabet $y_{ij}$, by $|v|_i$ we denote the total number of occurrences of letters $y_{11}^{\pm 1},\ldots, y_{ij_i}^{\pm 1}$ in $v$. For example, since $w$ is a word in the free generators of $N_{r,c}$, we have $|w|_1=|w|$; for $v=[x_1,x_2]=y_{2j}$ we have $|v|_1=0$ and $|v|_2=1$. Formally, $|w|_{\le 0}=0$ and $|w|_{>c}=|w|_c$. We rewrite the word $w$ in $c$ steps. On step $1\le n\le c$ we obtain an expression
\[
w=y_{11}^{\alpha_{11}}\cdots y_{nj_n}^{\alpha_{nj_n}}\cdot w_{n+1},
\]
where the left factor is the initial part of the desired right hand side of~\eqref{eq:malcev_coords} up to $y_i$ of weight $n$, will all involved $\alpha_{ij}$ satisfying $|\alpha_{ij}|\le P(|w|)$, and $w_{n+1}$ is a word in $y_{(n+1)1},\ldots,y_{cj_c}$ with
\begin{equation}\label{eq:collection_growth}
|w_{n+1}|_{n+\ell}\le C_n|w|^{n+\ell}\ \mbox{ for each }\ \ell=1,\ldots,c-n.
\end{equation}
Below we describe step $n$, assuming step $n-1$ is performed. The trivial step~$0$ serves as base of induction.

Thus, let $w^{(0)}=w_n$ be a word in $y_{n1},\ldots,y_{cj_c}$ as described above. Pick an occurrence of $y_{n1}^{\pm 1}$ and push it to the left using~\eqref{eq:commutator}, $w^{(0)}=y_{n1}^{\pm 1}w^{(1)}$. By~\eqref{eq:commutator} we have
\[
|w^{(1)}|_{n+\ell}\le |w^{(0)}|_{n+\ell}+C_0|w^{(0)}|_{\ell}.
\]
Proceed in the same way with all occurrences of $y_{n1}^{\pm 1}$, then all occurrences of $y_{n2}^{\pm 1}$, and so on, as follows. Given $w^{(i)}$, we express it as $w^{(i)}=y^{\pm1}_{nk_i}w^{(i+1)}$, where $\{k_i\}$ is a non-decreasing sequence and
\begin{equation}\label{eq:pascal}
|w^{(i+1)}|_{n+\ell}\le |w^{(i)}|_{n+\ell}+C_0|w^{(i)}|_{\ell}.
\end{equation}
After $N=|w_n|_{n}$ repetitions we have
\[
w_n=y_{n1}^{\alpha_{n1}}\cdots y_{nj_n}^{\alpha_{nj_n}}\cdot w^{(N)},
\]
where $|\alpha_{ni}| \le |w_n|_n\le C_n|w|^n$. Now we only have to verify that $w_{n+1}=w^{(N)}$ satisfies condition~\eqref{eq:collection_growth}. Indeed, by repeatedly applying~\eqref{eq:pascal} we see that
\begin{eqnarray*}
|w^{(N)}|_{n+\ell} & \le & |w^{(N-1)}|_{n+\ell}+C_0 |w^{(N-1)}|_\ell \\ 
&\le&
\left(|w^{(N-2)}|_{n+\ell}+C_0 |w^{(N-2)}|_\ell\right)+C_0\left(|w^{(N-2)}|_{\ell}+C_0 |w^{(N-2)}|_{\ell-n}
\right) \\
&=& 
|w^{(N-2)}|_{n+\ell}+2C_0 |w^{(N-2)}|_\ell+C_0^2 |w^{(N-2)}|_{\ell-n}\\
&\le& 
\ldots \\
&\le& 
{j\choose 0} |w^{(N-j)}|_{n+\ell}+ 
\ldots+{j\choose j}C_0^j|w^{(N-j)}|_{n+\ell-jn} \\
&=& 
\sum_{\iota=0}^{j}{j\choose\iota}C_0^\iota|w^{(N-j)}|_{n+\ell-\iota n}.
\end{eqnarray*}
Note that in the latter expression all terms with $n+l-\iota n<n$ are zero, so the sum is
\begin{eqnarray*}
\sum_{\iota=0}^{j}{j\choose\iota}C_0^\iota|w^{(N-j)}|_{n+\ell-\iota n}&=&\sum_{\iota\le\frac{\ell}{n}}{j\choose\iota}C_0^\iota|w^{(N-j)}|_{n+\ell-\iota n} \\
&\le&
\ldots \\
&\le& 
C_0^{c/n}\sum_{\iota\le\frac{\ell}{n}}{N\choose\iota}|w^{(0)}|_{n+\ell-\iota n}.
\end{eqnarray*}
By~\eqref{eq:collection_growth}, we have $|w^{(0)}|_{n+\ell-\iota n}=|w_n|_{n+\ell-\iota n}\le C_n|w|^{n+\ell-\iota n}$. Also recall that $N=|w_n|_{n}\le C_n|w|^n$, so
\begin{eqnarray*}
C_0^{c/n}\sum_{\iota\le\frac{\ell}{n}}{N\choose\iota}|w^{(0)}|_{n+\ell-\iota n}
&\le&
C_0^{c/n}\sum_{\iota\le\frac{\ell}{n}}N^\iota C_n|w|^{n+\ell-\iota n}\\
&\le& 
C_0^{c/n}\sum_{\iota\le\frac{\ell}{n}}(C_n|w|^{n})^\iota C_n|w|^{n+\ell-\iota n} \\ 
&\le&
C_0^{c/n}\sum_{\iota\le\frac{\ell}{n}}C_n^\iota C_n|w|^{\iota n}|w|^{n+\ell-\iota n} \\ 
&\le&
C_{n+1}|w|^{n+\ell}.
\end{eqnarray*}
Now, to establish the statement, it suffices to set the polynomial $P(x)=C_cx^c$.
\end{proof}

\begin{remark}
Note that we actually have shown that $|\alpha_{ij}|\le C_i|w|^i$ for some constant $C_i$ that only depends on the class of nilpotency $c$.
\end{remark}

\begin{proposition}\label{pr:permutation}
For every integer $c>0$, there is a polynomial $P$ such that in any nilpotent group $H$ of nilpotency class at most $c$ for any $k$ elements $f_1,\ldots,f_k\in H$, any $\varepsilon_1,\ldots,\varepsilon_k\in\{0,1\}$, and any permutation $\tau\in S_k$, an equality
\[
f_{\tau_1}^{\varepsilon_{\tau_1}}\cdots f_{\tau_k}^{\varepsilon_{\tau_k}}=f_1^{\varepsilon_1}\cdots f_k^{\varepsilon_k}\cdot h_1^{\alpha_1}\cdots h_{m}^{\alpha_m}
\]
takes place, where $h_1,\ldots,h_m$ are all possible iterated commutators of $f_1,\ldots,f_k$ of weights $2,3,\ldots,c$, and $|\alpha_i|\le P(k)$, $i=1,\ldots,m$.
\end{proposition}
\begin{proof}
Consider the free nilpotent group $N_{k,c}$ on $k$ generators $x_1,\ldots,x_k$. By Theorem~\ref{th:nilp_growth}, the equality
\[
x_{\tau1}^{\varepsilon_{\tau1}}\cdots x_{\tau k}^{\varepsilon_{\tau k}}=x_1^{\varepsilon_1}\cdots x_k^{\varepsilon_k}\cdot y_{k+1}^{\alpha_{k+1}}\cdots y_{m}^{\alpha_m}
\]
takes place, where $y_{k+1},\ldots,y_m$ are iterated commutators of $x_i$, and $|\alpha_i|\le P(k)$. Since $N_{k,c}$ is a free nilpotent group of class $c$ on free generators $x_1,\ldots, x_k$, the same equality holds under the specialization $x_i\to f_i$.
\end{proof}

\section{Distortion in polycyclic groups}

The following two statements are well known, but we provide the proof here for the sake of completeness. Recall that $\|\cdot\|$ denotes the Euclidean norm. Also note that below we follow the convention that, for $m,l\in\mathbb N\cup\{0\}$,  the binomial coefficient $m\choose l$ is $0$ whenever $m<l$.

\begin{lemma}\label{le:power_of_matrix}
Let $M$ be an $n\times n$
matrix with complex entries, let $\alpha$ be the maximum of absolute values of its eigenvalues. There is a positive constant $C_M$ such that for any $v\in\mathbb C^n$ and any $k\in\mathbb N$:
\[
\|M^kv\|\le C_M\left(\alpha^k+{k\choose 1}\alpha^{k-1}+\ldots+{k\choose n}\alpha^{k-n}\right)\|v\|.
\]
\end{lemma}
\begin{proof}
Let $M=C^{-1}J_MC$, where $J_M$ is the Jordan normal form of $M$ and $C\in\mathrm{GL}_n(\mathbb C)$. Let $b,c>0$ be such that $\|Cv\|\le b\|v\|$ and $\|C^{-1}v\|\le c\|v\|$ for all $v\in\mathbb C^n$. Then, inspecting Jordan normal form, one can see that
\begin{eqnarray*}
\|M^kv\|=\|C^{-1}J_M^kCv\|&\le& c\|J_M^kCv\| \\
&\le& c \left(\alpha^k+{k\choose 1}\alpha^{k-1}+\ldots+{k\choose n}\alpha^{k-n}\right)\|Cv\| \\
&\le& bc \left(\alpha^k+{k\choose 1}\alpha^{k-1}+\ldots+{k\choose n}\alpha^{k-n}\right)\|v\|,
\end{eqnarray*}
as required.
\end{proof}

\begin{proposition}\label{pr:big_eigenvalue}
Let $H=\langle x\rangle\ltimes K$, where $K=\mathbb Z^n$ and $x$ acts on $K$ by conjugation via a matrix $X\in\SL_n(\mathbb Z)$. Suppose $H$ is not virtually nilpotent. Then $X$ has a complex eigenvalue of absolute value greater than $1$.
\end{proposition}
\begin{proof}

Suppose all eigenvalues of $X$ are of absolute value at most~$1$.

Let $e_1,\ldots,e_n$ be the standard basis for $K=\mathbb Z^n$ and suppose 
a word $w=w(x,e_1,\ldots,e_n)$ represents an element $g$ of $K$. 
Then algebraic sum of the exponents of $x$ in $w$ is $0$ and, hence, $w$ can be written as:
\begin{equation}\label{eq:conj_form}
w=g_1^{x^{k_1}}\cdots g_\ell^{x^{k_\ell}},
\end{equation}
where $\ell\le |w|$, $|k_i|\le |w|$, $g_i\in K$ for each $i=1,\ldots,\ell$, and  $|g_1|+\ldots+|g_\ell|\le |w|$. 
Recall that for an element $g\in K=\mathbb Z^n$, one has 
$$
\|g\|\le |g|_K\le \sqrt n\|g\|,
$$
where $|g|_K$ stands for the word length of $g$ in standard generators of $K$. 
Since $\alpha\le 1$, it follows from Lemma~\ref{le:power_of_matrix} that there is a polynomial $p(k)$ 
of degree $n$ such that $\|X^kg\|\le p(k)\|g\|$ for every $g\in K$. We therefore obtain:
\[
|g_i^{x^{k_i}}|_K\le \sqrt n\|g_i^{x^{k_i}}\|=\sqrt n\|X^{k_i}g_i\|\le\sqrt n\, p(k_i)\|g_i\|\le \sqrt n\,p(k_i)|g_i|_K,
\]
and, hence:
\begin{align*}
|w|_K=|g_1^{x^{k_1}}\cdots g_\ell^{x^{k_\ell}}|_K  &\le \sum_{i=1}^\ell |g_i^{x^{k_i}}|_K
\le \sum_{i=1}^\ell \sqrt{n} p(k_i) |g_i|_K\\
&\le \sqrt{n} p(|w|)\sum_{i=1}^\ell|g_i|_K \le \sqrt n\, p(|w|)|w|.
\end{align*}
Therefore, the distortion of $K$ in $H$ is at most polynomial, say $|w|_K\le q(|w|)$.

Finally we claim that $H$ has polynomial growth. 
Indeed, if $u\in H$ is an element of the ball $B_N$ of radius $N$ 
in the group $H$ in generators $x,e_1,\ldots,e_n$, then
$u=x^kw$, where $w$ is of the form~\eqref{eq:conj_form} and $|k|\le N$. 
Therefore:
\[
|B_N|\le (2N+1)\cdot (2q(N)+1)^n,
\]
and $H$ has polynomial growth.
Thus $H$ is virtually nilpotent by~\cite{Gromov_pgrowth:1981}.
\end{proof}


Recall that every polycyclic group $G$ has a unique maximal normal nilpotent subgroup, called the Fitting subgroup of $G$ and denoted $\mathop{\mathrm{Fitt}}G$ (see, for example, \cite[Chapter~1]{segal2005polycyclic}). Note that $\mathop{\mathrm{Fitt}}G$ is a characteristic subgroup of $G$. Let
\begin{equation}\label{eq:poly_series}
1=G_0\lhd G_1\lhd G_2 \lhd \ldots\lhd G_m=G
\end{equation}
be a subnormal series for $G$ with cyclic quotients. Denote $\mathop{\mathrm{Fitt}}G_i=H_i$. For each $i=0,\ldots, m-1$, we have that $G_i\lhd G_{i+1}$ and $H_i$ is a characteristic subgroup of $G_{i}$, therefore, $H_i\lhd G_{i+1}$. If follows that $H_i\le H_{i+1}$. Suppose that $H=\mathop{\mathrm{Fitt}}G$ is a term of the polycyclic series~\eqref{eq:poly_series}, $H=G_j$ with $j<m$. Then $H=H_j\le H_{j+1}\le H_m=H$, so $H_j=H_{j+1}$. Observe that  in this case, $G_{j+1}$ is not virtually nilpotent if $G_{j+1}/G_j$ is infinite cyclic. Indeed, if $G_{j+1}$ is virtually nilpotent, then it has a finite index normal nilpotent subgroup and therefore $H_{j+1}>H_j$.

\begin{proposition}\label{pr:abelian-by-cyclic-inside}
Let $G$ be a polycyclic group that is not vitually nilpotent. 
There exists an element $x\in G$ and normal nilpotent subgroups $K\le H$ of $G$
such that $H/K$ is infinite abelian and $\langle x,H\rangle/K$ is not virtually nilpotent.
\end{proposition}
\begin{proof}
Since $G/\mathop{\mathrm{Fitt}}G$ is infinite, it has a polycyclic series~\eqref{eq:poly_series} such that $\mathop{\mathrm{Fitt}}G=G_j$, $j<m$, and $G_{j+1}/G_j$ is infinite cyclic (see, for example,~\cite[Chapter~1, Proposition~2]{segal2005polycyclic}). Let $g_{j+1}\in G_{j+1}$ be such that $G_{j+1}=\langle g_{j+1},G_j\rangle$. We claim that we can take $x=g_{j+1}$, $H=G_j$, and $K$ to be the commutator subgroup $H'=[H,H]$ of $H$.

Indeed, the subgroup $K=H'$ is characteristic in $H=G_{j+1}=\mathop{\mathrm{Fitt}}G$, therefore normal in $G$. Further, if the abelianization $H/H'$ is finite, then $H$ is finite (for example, by~\cite[Chapter~1, Corollary~9]{segal2005polycyclic}) and therefore so is $G$, by~\cite[Chapter~1, Lemma~6]{segal2005polycyclic}. 
It was observed above that $G_{j+1}=\langle x, H\rangle$ is not virtually nilpotent. Finally, if $\langle x, H\rangle/H'$ is virtually nilpotent, it follows by~\cite[Chapter~1, Corollary~12]{segal2005polycyclic}) that $\langle x,H\rangle$ is virtually nilpotent, which is not the case.
\end{proof}


For a polycyclic group $G$ let $x\in G$ and $K\le H\le G$ be as provided by
Proposition~\ref{pr:abelian-by-cyclic-inside}. 
In Section~\ref{sec:abelian_by_cyclic}, we show that $\SSP$ in the abelian-by-cyclic group $\langle x,H\rangle/K$ is $\NP$-hard. In Section~\ref{sec:general_case}, we show that the instances involved in this reduction, in turn, polynomially reduce to $\SSP(\langle x,H\rangle)$ and therefore to $\SSP(G)$. Together with the observation that the word problem in $G$ is solvable in polynomial time, this will imply that $\SSP(G)$ is $\NP$-complete.



\section{$\SSP$ in abelian-by-cyclic groups}\label{sec:abelian_by_cyclic}

Fix a group $F = \MZ\ltimes \MZ^n$ with exponentially distorted $\MZ^n$ by a matrix $X\in\SL_n(\MZ)$.
Also, fix a generating set $\{x,e_1,\ldots,e_n\}$, where
$x$ is the generator of $\MZ$ and $e_1,\ldots,e_n$ are standard generators
for $\MZ^n$. Let $\varphi:F(x,e_1,\ldots,e_n)\to F$ be the canonical
epimorphism.
Below we reduce a problem known to be $\NP$-complete, namely zero-one equation problem, to $\SSP(F)$.


As before, let $\alpha$ be the greatest absolute value of
an eigenvalue for $X\in\SL_n(\MZ)$. Define a polynomial $p(k)\in\mathbb R[k]$:
\[
p(k) =C_X\cdot \left(1+{k\choose 1}\frac{1}{\alpha}+\ldots+{k\choose n}\frac{1}{\alpha^n}\right),
\]
where $C_X$ is a constant provided by Lemma~\ref{le:power_of_matrix}.

\begin{proposition}\label{pr:norm_growth} 
In the above notation, for every $k\in\mathbb N$, there is $j\in\{1,\ldots, n\}$ 
satisfying:
\[
\tfrac{1}{\sqrt n} \alpha^k\le \|X^k e_j\|\le p(k)\alpha^k.
\]
\end{proposition}

\begin{proof} 
The latter inequality is established in Lemma~\ref{le:power_of_matrix}. 
To show the former inequality, observe that if for some $\mu>0$:
\[
\|X^k e_i\|< \mu\quad \mbox{for each } i=1,\ldots,n,
\]
then for any $v=v_1e_1+\ldots+v_ne_n\in\mathbb C^n$ we have
\[
\|X^kv\|\le \sum_{i=1}^n|v_i|\|X^ke_i\|< \mu\sum_{i=1}^n|v_i|\le \mu \sqrt n\|v\|.
\]
Since there is $0\neq v\in\mathbb C^n$ such that $\|X^kv\|=\alpha^k\|v\|$, it follows that $\alpha^k< \mu\sqrt n$, that is $\mu>\tfrac{1}{\sqrt n}\alpha^k$, as required.
\end{proof}

By Proposition~\ref{pr:big_eigenvalue}, $\alpha>1$.
Observe that given $k\in\mathbb N$, one can find 
a basis vector $e$ (denoted by $e_k^\ast$)
provided by Proposition~\ref{pr:norm_growth} 
in polynomial time by computing $X^ke_1,\ldots,X^ke_n$. 
Now, for $\lambda\ge 1$ and $k\in\MN$ define a constant:
\[
c_{\lambda,k}=\lceil\log_\alpha(p(k))\rceil
+\lceil\log_\alpha \lambda \rceil+ \lceil \log_\alpha\sqrt n\rceil+1
\]
and notice that:
\[
\lambda\|X^{k}e_{k}^\ast\| <
\|X^{k+c_{\lambda,k}} e_{k+c_{\lambda,k}}^\ast\|.
\]
For a sequence of numbers
$n_1=1$, $n_{i+1}=n_i+c_{\lambda,n_i}$ we have:
\begin{equation}\label{eq:growth}
\lambda^{k-1}\|X^{n_1}e_{n_1}^\ast\|<
\lambda^{k-2}\|X^{n_2}e_{n_2}^\ast\|<
\ldots<
\lambda\|X^{n_{k-1}}e_{n_{k-1}}^\ast\|<
\|X^{n_k}e_{n_k}^\ast\|.
\end{equation}
Denote the words corresponding to 
$X^{n_{1}}e_{n_{1}}^\ast, \ldots,X^{n_{k}}e_{n_{k}}^\ast$
by $w_{\lambda,1},\ldots,w_{\lambda,k}$, i.e., define:
\[
w_{\lambda,i} = x^{-n_i} e_{n_i}^\ast x^{n_i}.
\]
Clearly, $|w_{\lambda,i}| \le 1+2n_i$. Now we find an upper bound for $n_i$. Notice that $c_{\lambda,k}\le 
A+B\log_\alpha(\lambda k)$,
where the constants $A,B>0$ depend only on $\alpha$, $C_X$, and $n$. 
Then
\[
n_i \le  n_1+iA+B(\log_\alpha(\lambda n_1)+\cdots+\log_\alpha(\lambda n_i))\le iA+iB\log_\alpha(\lambda n_i),
\]
or
\[
\frac{n_i}{A+B\log_\alpha(\lambda n_i)}\le i,\ \mbox{ that is, }\ \frac{\lambda n_i}{A+B\log_\alpha(\lambda n_i)}\le \lambda i.
\]
Since there is a constant $C\ge 0$ (that depends only on $A$, $B$, $\alpha$) such that $\frac{t}{A+B\log_\alpha(t)}\ge \sqrt{t}-C$ for all $t\ge 1$, we have
\[
\sqrt{\lambda n_i}-C\le \lambda i, \mbox{ so } n_i\le \lambda^{-1}(\lambda i+C)^2.
\]
Therefore,
\begin{equation}\label{eq:ckl}
|w_{\lambda,i}|\le 1+2n_i\le 1+2\lambda^{-1}(\lambda i+C)^2,
\end{equation}
where $C$ ultimately depends only on $X$ and $n$. (Of course, better estimates for the growth of $n_i$ are possible but immaterial for our purposes.)


\begin{proposition}\label{prop:abelian_np_hard}
$\SSP(F)$ is $\NP$-hard.
\end{proposition}

\begin{proof}
For an instance of $\ZOE$:
\[
A=
\left[
\begin{array}{ccc}
a_{11} & \ldots & a_{1k} \\
\vdots & &\vdots \\
a_{k1} & \ldots & a_{kk} \\
\end{array}
\right]
\]
choose $\lambda=k$ and consider the instance $(g_1,\ldots,g_k,g)$ of $\SSP(F)$, where:
\[
g_i = w_{\lambda,1}^{a_{1i}}\ldots w_{\lambda,k}^{a_{ki}}
\ \mbox{ and }\ 
g = w_{\lambda,1} \ldots w_{\lambda,k}.
\]
We claim that the instance of $\ZOE$ is positive if and only if
the corresponding instance of $\SSP$ is positive. Indeed, let $k\ge 2$ (the case $k=1$ is immediate). The instance of $\SSP$ is positive if and only if the linear combination
\begin{equation}\label{eq:linear_comb}
\left(-1+\sum_{i=1}^k a_{1i}\varepsilon_i\right)X^{n_1}e^\ast_{n_1}+\ldots+ \left(-1+\sum_{i=1}^k a_{ki}\varepsilon_i\right)X^{n_k}e^\ast_{n_k}
\end{equation}
is equal to $0$ for some $\varepsilon_i\in\{0,1\}$. Since for every coefficient in~\eqref{eq:linear_comb} we have
\[
-1\le -1+\sum_{i=1}^k a_{ji}\varepsilon_i\le k-1,
\]
it follows from~\eqref{eq:growth} that~\eqref{eq:linear_comb} is trivial if and only if all coefficients are $0$, i.e., there are $\varepsilon_i\in\{0,1\}$ such that
\[
\sum_{i=1}^k a_{ji}\varepsilon_i=1\ \mbox{ for every }\ 1\le j\le k.
\]
The latter is precisely the condition for the corresponding instance of $\ZOE$ to be positive.

 Furthermore, since it is straightforward
to write $w_{\lambda,k}$, the time to generate the instance of $\SSP$
is proportional to 
\begin{align*}
|g_1|+\ldots+|g_k|+|g| &\le (k+1)|g|\le (k+1)(k|w_{k,k}|).
\end{align*}
Taking~\eqref{eq:ckl} into account, we see that the above is clearly polynomial in~$k$.
Therefore, $\ZOE$ can be reduced to $\SSP(F)$ in polynomial time.
Thus, $\SSP(F)$ is $\NP$-hard.
\end{proof}

Since the word problem in $F$ is decidable in polynomial time, the 
following holds:

\begin{corollary}
$\SSP(F)$ is $\NP$-complete.
\end{corollary}

\begin{remark}\label{rem:they_are_in_H}
Note that the elements $w_{\lambda,k}$ involved in the above reduction belong to the ``bottom'' subgroup $\mathbb Z^n$ of $F=\MZ\ltimes \MZ^n$.
\end{remark}

\section{$\SSP$ in polycyclic groups}\label{sec:general_case}
Now we turn to the case of an arbitrary polycyclic group $G$. 

\begin{theorem}\label{th:poly-np-complete}
Let $G$ be non-virtually nilpotent polyclic group.
Then $\SSP(G)$ is $\NP$-complete.
\end{theorem}
\begin{proof}
Since word problem in $G$ is polynomial time decidable, it suffices to show that $\SSP(G)$ is $\NP$-hard. For that, we show that the reduction of $\ZOE$ to the subset sum problem in an abelian-by-cyclic group described in Section~\ref{sec:abelian_by_cyclic} can be refined to deliver a reduction to $\SSP(G)$.

Let $x\in G$ and normal subgroups $K\le H$ of $G$ be as provided by Proposition~\ref{pr:abelian-by-cyclic-inside}. Passing to a subgroup of $H/K$, we may assume that the group $F=\langle x, H\rangle/K$ is (free abelian)-by-cyclic, as specified in Section~\ref{sec:abelian_by_cyclic}.

Let a $k\times k$ matrix $Z=(a_{ji})$ be given as an input of $\ZOE$, let $g'_1,\ldots,g'_k,g'\in F$ be the equivalent input of $\SSP(F)$, and let $w'_{\lambda,1},\ldots, w'_{\lambda,k}$, with $\lambda=k$, be the corresponding elements of $F$ involved in the construction, as chosen in the proof of Proposition~\ref{prop:abelian_np_hard}:
\[
g'_i = {w'_{\lambda,1}}^{a_{1i}}\ldots {w'_{\lambda,k}}^{a_{ki}}
\ \mbox{ and }\ 
g' = w'_{\lambda,1} \ldots w'_{\lambda,k}.
\]
Fix representatives $w_{\lambda,k}\in\langle x, H\rangle$ of $w'_{\lambda,k}$: $w'_{\lambda,1}=w_{\lambda,1}K,\ \ldots,\  w'_{\lambda,k}=w_{\lambda,k}K$. Consequently, fix representatives of $g'_i,g'$ as follows:
\[
g_i = {w}_{\lambda,1}^{a_{1i}}\ldots {w}_{\lambda,k}^{a_{ki}}
\ \mbox{ and }\ 
g = w_{\lambda,1} \ldots w_{\lambda,k}.
\]
Note that we may assume that elements in $\langle x,H\rangle/K$ are encoded by words in generators of $\langle x,H\rangle$, so we can neglect the time required to choose elements $w_{\lambda,i},g_i,g$.

By construction, if the equality
\[
{g_1'}^{\varepsilon_1}\cdots {g_k'}^{\varepsilon_k}=g'
\]
takes place in $F$, then
\[
\sum_{i=1}^k a_{ji}\varepsilon_i=1\ \mbox{ for every }\ 1\le j\le k,
\]
or, in other words, each factor $w'_{\lambda,i}$ occurs in the product ${g_1'}^{\varepsilon_1}\cdots {g_k'}^{\varepsilon_1}$ exactly once. By the choice of $w_{\lambda,i}$ and $g_i$, the same is true for factors $w_{\lambda,i}$ in the product $g_1^{\varepsilon_1}\cdots g_k^{\varepsilon_1}$, that is, the latter and the element $g$ are products of the same factors $w_{\lambda,1},\ldots, w_{\lambda,k}$ in, perhaps, different orders.

Recall that $w_{\lambda,i}\in H$ by Remark~\ref{rem:they_are_in_H} and therefore generate a nilpotent subgroup $H_0$ of $H$.
Let $c_0\le c$ be the nilpotency classes of $H_0\le H$, respectively.
Since $H_0$ is nilpotent, the $w_{\lambda,i}$ factors in the product $g_1^{\varepsilon_1}\cdots g_k^{\varepsilon_k}$
can be rearranged as
\[
g_1^{\varepsilon_1}\cdots g_k^{\varepsilon_k}\cdot h_1^{\alpha_1}\cdots h_m^{\alpha_m}=
w_{\lambda,1}\cdots w_{\lambda,k}=g,
\]
where $h_1,\ldots,h_m$ are iterated commutators of $w_{\lambda,1},\ldots,w_{\lambda,k}$ up to weight $c_0\le c$. Since there are $k$ factors $w_{\lambda,i}$ in the product $g_1^{\varepsilon_1}\cdots g_k^{\varepsilon_k}$, by Proposition~\ref{pr:permutation}, there is a polynomial $P$ that only depends on $c$ such that each $|\alpha_1|,\ldots,|\alpha_m| $ can be taken to not exceed $P(k)$.

Therefore, if the instance $g_1',\ldots,g_k',g'$ of $\SSP(F)$ is positive, then the instance
\begin{equation}\label{eq:new_inst}
g_1,\ldots,g_k,\underbrace{h_1,\ldots,h_1}_{P(k)},\underbrace{h_1^{-1},\ldots,h_1^{-1}}_{P(k)},\ldots, \underbrace{h_m,\ldots,h_m}_{P(k)},\underbrace{h_m^{-1},\ldots,h_m^{-1}}_{P(k)}, g
\end{equation}
of $\SSP(G)$ is positive. The converse is immediate since $H/K$ is abelian and therefore $h_i\in K$. It follows that the instance $Z$ of $\ZOE$ is positive if and only if~\eqref{eq:new_inst} is a positive instance of $\SSP(G)$.

It is only left to observe that there are at most $(2k)^c+(2k)^{c-1}+\ldots +(2k)^2$ iterated commutators of $2k$ elements $w_{\lambda,i}^{\pm1}$, each of length at most $2^c |w_{\lambda,k}|$, so the tuple~\eqref{eq:new_inst} can be constructed in time which is polynomial in the size of the original input $Z$. By Section~\ref{sec:abelian_by_cyclic}, this gives a reduction of $\ZOE$ to $\SSP(G)$. 
\end{proof}

\begin{corollary}
Let $G$ be a polycyclic group. If $G$ is virtually nilpotent,
then $\SSP(G) \in\P$.
Otherwise $\SSP(G)$ is $\NP$-complete.
\end{corollary}
\begin{proof} If $G$ is virtually nilpotent, the subset sum problem has polynomial time solution by~\cite[Theorem 3.3]{Miasnikov-Nikolaev-Ushakov:2014a}. If $G$ is not virtually nilpotent, the statement immediately follows by Theorem~\ref{th:poly-np-complete}.
\end{proof}

Therefore, $\SSP(G)$ is $\NP$-hard for any group $G$ that contains a not virtually nilpotent polycyclic subgroup and $\NP$-complete if, additionally, the word problem for $G$ is polynomial time decidable. In particular, the subset sum problem is $\NP$-complete for any virtually polycyclic group that is not virtually nilpotent.

\section{Open questions}\label{sec:questions}

In conclusion, we would like to formulate several open problems related to our work.
First, given that $\SSP(G)$ is polynomial time decidable in virtually nilpotent groups,  $\NP$-complete in otherwise polycyclic groups by Theorem~\ref{th:poly-np-complete}, $\NP$-complete in $BS(1,2)$ and free metabelian groups by~\cite{Miasnikov-Nikolaev-Ushakov:2014a}, and $\NP$-complete in lamplighter group by~\cite{Misch-Tr:2016b}, it is only natural to make the following conjecture:

\medskip\noindent{\bf Conjecture.} Let $G$ be a finitely generated 
(non virtually nilpotent) metabelian group. Then $\SSP(G)$ is $\NP$-complete.

\medskip
The next logical step is to study the subset sum problem in solvable groups. Taking into account that not all solvable groups have decidable word problem, we can state the following question.

\medskip\noindent{\bf Question.} Let $G$ be a finitely generated (non virtually nilpotent) solvable group with polynomial time decidable word problem. Is it true that $\SSP(G)$ is $\NP$-complete?

\medskip
Third, while not immediately related to the present paper, the following question is interesting and still open.

\medskip
\noindent{\bf Question.} What is the complexity of the subset sum problem in the first Grigorchuk's group?

\bibliography{main_bibliography}

\end{document}